\documentclass{amsart}
\usepackage{color}
\usepackage{amsmath,amssymb,amsthm}
\usepackage[margin=1in]{geometry}
\usepackage{amssymb,txfonts,xspace}
\usepackage{graphicx}
\usepackage{mathtools}
\usepackage[normalem]{ulem}

\usepackage{tikz}

 \usetikzlibrary{arrows}
 \usetikzlibrary{decorations.pathmorphing}
 \usetikzlibrary{shapes}
 \usetikzlibrary{decorations.markings}	

\newcommand{\Z}{\mathbb Z}

\newcommand{\R}{\mathbb R}
\renewcommand{\H}{\mathbb H}
\newcommand{\calP}{\mathcal {P}}
\newcommand{\calG}{\mathcal {G}}
\newcommand{\HNN}{\Z^m*_{g \mapsto g^3}}

\def\Tree{{\mathrm {Tree}}}

\newcommand{\Sm}[1] {\mathbb{S}_{\Z^{#1}}}

\newtheorem{thm}{Theorem}
\newtheorem{lem}[thm]{Lemma}
\newtheorem{prop}[thm]{Proposition}

\newtheorem{cor}[thm]{Corollary}
\newtheorem{remark}[thm]{Remark}
\newtheorem{defn}[thm]{Definition}

\let\phi\varphi

\begin{document}

\title{Growth in higher Baumslag-Solitar groups}
\author{Ayla P. S\'anchez and Michael Shapiro}

\begin{abstract}
We study the HNN extension of $\Z^m$ given by the cubing endomorphism
$g\mapsto g^3$, and prove that such groups have rational growth with
respect to the standard generating sets.  We compute the subgroup
growth series of the horocyclic subgroup $\Z^m$ in this family of
examples, prove that for each $m$ the subgroup has rational growth.
We then use the tree-like structure of these groups to see how to
compute the growth of the whole group.  
In the appendix, the subgroup
growth series has been computed for all $m \leq 10$.  
\end{abstract}

\maketitle

\section{Introduction} \label{sec-intro}


In this paper we will study the HNN extension of $\Z^m$ by the
expanding map $\varphi(g)=g^3$ that cubes all elements.  When $m=1$,
this gives the well-known Baumslag-Solitar group $BS(1,3)$.  The
groups $\HNN$ can be thought of as graphs of groups given by a loop on
a single vertex labeled by $\Z^m$.  These groups are solvable but not
nilpotent: the commutator subgroup, and indeed each term in the lower
central series, is isomorphic to $\Z^m$.  Their model spaces are
combinatorially $\Tree \times \Z^m$, and the tree admits a height
function which gives a horofunction $\tau$ on the group.  Its level
sets are countably many copies of $\Z^m$. We refer to $\Z^m=\langle
a_1,\dots,a_m\rangle$ as the {\em horocyclic subgroup} of $\HNN$.
Below, we define and study the subgroup growth series $\Sm{m} (x)$ for
the horocyclic subgroups of these cubing extensions, with standard
generators $\langle a_1,\dots,a_m,t\rangle$.  Our main result 
will be to show the rationality of the group growth series via the
growth of the horocyclic subgroup.  Rationality of growth series for
groups has been studied intensively since at least the 1980s, and
subgroup growth has been a key tool, as explained in the introduction
of the important paper \cite{Stoll}.

Rationality for the $m=1$ case follows from the results of Collins,
Edjvet, and Gill in \cite{CEG}, and the growth series is explicitly
computed by Freden, Knudson, and Schofield in \cite{FKS-GBSG1} and
discussed in further detail by Freden in \cite{FredenOHGGT}, where the
author also outlines the computation in the $m=2$ case.  Using the
different approach we take here, we recover these results and produce
formulas for the subgroup growth for all $m$.

To summarize our results, we need to make a few definitions first. 
 We will establish in Theorem \ref{LanguageTheorem} that with the
exception of the single group element $a_1 \ldots a_m$, all words in
the positive orthant of $\Z^m$ can be expressed in a geodesic normal
form $t^n v s$, where $v$ comes from a finite list of strings called
{\em caps} and $s \in (t^{-1}W)^n$, where $W$ is also finite. 
 We refer to pairs $(t^n, s)$ as {\em prefix/suffix
  pairs}.  We then define the {\em prefix/suffix series} $R_m(x)$ to
be the power series where the coefficient of $x^n$ is the number of
prefix/suffix pairs whose combined length is $n$.  Similarly, we
define the {\em cap polynomial} $V_m(x)$ to be the power series where
the coefficient of $x^n$ is the number of caps of length $n$, polynomial
becasue $V$ is finite. 
Putting this information together, we can build the growth series of
the positive orthant of $\Z^m$ in $\Z^m*_{g \mapsto g^3}$, which we
call the {\em positive series} $P_m(x)$, and we finally patch those
together to build $\Sm{m}(x)$.

These form a recursive system and allow us to compute the subgroup growth explicitly, 
which directly shows the our first main result:
\begin{thm} \label{allrational}
For each $m$,  the subgroup growth series of  $\Z^m < \Z^m *_{g \mapsto g^3}$ with standard generators is rational.
\end{thm} 

Using this, we then build a tree of cosets of $\Z^m$, and show the following:

\begin{thm} \label{grouprational}
The growth series of $\Z^m *_{g \mapsto g^3}$ with standard generators is rational for all $m$.
\end{thm}

\subsection*{Acknowledgments}
The authors would like to thank Moon Duchin, Murray Elder, and Meng-Che Ho. 
MS would like to thank Karen Buck for helping to jump-start some of the neurons used in this work.

\section{Background} \label{sec-background}

\subsection{Spellings and words}

When working with group presentations, we will denote the empty letter
as $\epsilon$, generators with lowercase letters, and inverses of
generators with capital letters.  In the case of $m \leq 3$, we will
write our generators as $a, b,$ and $c$ for simplicity.  
We write $\ell(g)$ to denote spelling length of an expression and 
$\ell_{\calG}(g)$ to denote word length of a group element (which is the minimal spelling length amongst all representatives of a group element).
For example, in $BS(1,3)$, $\ell (a^6) = 6$, while $\ell_{\calG}(a^6) = 4$. 


\subsection{Growth}

For a group $G$ with generating set $\calG$, we denote the sphere of radius $n$ 
as $S_n := \{ g \in G \mid \ell_{\calG}(g) = n \}$, that is, the set of all group elements whose geodesic spelling length is $n$.
Define then the {\em growth function} to be $\sigma(n) =\sigma_{G,\calG}(n)= \# S_n$.
Under the equivalence where two functions $f$ and $g$ are equivalent if there is some $A>0$ such that $g \leq Af(Ax)$ and $f \leq Ag(Ax)$,
the equivalence class of $\sigma$ is a group invariant, which in particular means it does not depend on generating set.
Growth of groups in this context is well studied, in particular it has been long known that a group is virtually nilpotent if and only if it has polynomial growth, 
non-elementary hyperbolic groups have exponential growth, and there exist groups of intermediate growth.
(For an in-depth treatment of the subject,  see \cite{Mann}.)
Among groups of exponential growth, you can define the growth order to be $\omega(G) = \lim \sigma(n)^{1/n}$,
and in fact exponential growth is characterized by $\omega(G) > 1$

From this we can define a slightly different way to measure how a group grows via a power series with coefficients $\sigma(n)$.
Define the {\em (spherical) growth series of $G$  relative to $\calG$} to be $\mathbb{S}(x) =\mathbb{S}_{G,\calG}(x)= \sum_{n=0}^{\infty} \sigma(n) x^n$.
One immediate observation about the growth series is that the radius of convergence for the growth series is precisely $1/\omega(G)$.

\subsection{Rationality}

The question that most often arises when studying growth series is {\em when is the growth series equal to a rational function?}
When a group with a generating set has this property, we say that $(G,\calG)$ has {\em rational growth}.
One important thing to notice is that unlike the rate of  growth, rationality can depend on generating set.
For many years, only hyperbolic groups and virtually abelian groups were known to have rational growth in all generating sets.
Recently,  Duchin and Shapiro proved the same always-rational property for the integer Heisenberg group in \cite{DS},
 but that is currently the only other group for which this is known to be true. 
On the other hand, groups of intermediate growth and recursively presented groups with unsolvable word problem\footnote{An anonymous referee pointed out to MS that the obvious proof here requires the group to be recursively 
presented.  This leaves open the question as to whether there is a group with unsolvable word problem and rational growth. 
An answer either way would be fascinating.}  are easily seen to  not have rational growth for any generating set.
In between those sit some automatic groups, solvable Baumslag-Solitar groups, and Coxeter groups, which all have rational growth with their standard generating set, 
as well as a roster of other groups and families that have rational growth for some generating set.
The higher Heisenberg groups are a family of two-step nilpotent groups which where shown by Stoll in \cite{Stoll} to have
irrational growth in their standard generators and rational growth in another generating set, 
and they provide the only currently known examples of this kind.

In addition to providing a classification of groups and generating sets, information about a group can be read off of the growth series.
In particular, if one has a growth series explicitly as a rational function, then the denominator yields the coefficients of a recurrence relation for the 
function $\sigma(n)$  for large $n$. 
Furthermore, rationality of the spherical growth series is equivalent to rationality of the cumulative growth series (taking instead $\beta(n)$ the number of elements in the ball of radius $n$),
for the simple reason that $\frac{\mathbb{S}(x)}{1-x}  = \mathbb{B}(x)$.
Further information can be read off of the denominator.
For example, $\mathbb{S}(x)$ (or $\mathbb{B}(x)$) has a pole of norm $< 1$ iff the group has exponential growth.

\subsection{Subgroup Growth}
In some cases, we are concerned with how a subgroup grows inside of
another group.  For this, given $H \le G$ , we look at the intersection of $S_n$ and $H$,
and then define $\sigma_H(n) = \# ( S_n \cap H).$ From this we define
the {\em (spherical) subgroup growth series of $H$ in $G$ relative to
  $\calG$ } as $\mathbb{S}_{H \leq G}(x) = \sum_{n=0}^{\infty} \sigma_H (n) x^n$.
$\sigma_H(n)$ captures information on how $H$ is embedded in $G$, for
example, $\sigma_\Z(n)$ for $\Z < BS(1,3)$, as we will see, has
exponential growth as its subgroup growth series has a pole inside the
unit disk, meaning lengths are heavily distorted by the embedding.
Furthermore, we often get information about $\mathbb{S}_G$ from
$\mathbb{S}_{H \leq G}$ because the growth in cosets sums to the full
growth series.  For example, Freden, Knudson, and Schofield build the
growth series of $BS(1,3)$ by finding the subgroup growth of $\Z$ in
\cite{FKS-GBSG1}.

Growth in subgroups can also be evidence of more interesting behavior
as well.  In \cite{Shapiro}, MS considered a central extension of the
surface group of genus 2, and proved that the subgroup $\langle z
\rangle$ grows rationally in $G = \langle x_1, x_2, y_1, y_2, z \mid z
= [x_1, y_1][x_2, y_2] \text{ central} \rangle$.  If we omit $z$ as a
generator, though, $G= \langle x_1, x_2, y_1, y_2 \mid [x_1, y_1][x_2,
  y_2] \text{ central} \rangle$ the central subgroup grows
transcendentally irrationally.  This was early evidence that
rationality depends on generating set, which was ultimately shown by
Stoll's higher Heisenberg group example in \cite{Stoll}.

\subsection{Baumslag-Solitar Groups}
The Baumslag-Solitar groups are the one-relator groups 
$$BS(p, q) = \langle a, t \mid ta^p T = a^q \rangle.$$
These groups have Cayley graphs that are built out of ``bricks" of the form 
\raisebox{-.13 in}{\begin{tikzpicture}[scale=.4] 
\begin{scope}[decoration={markings,  mark=at position 0.55 with {\arrow{>}}} ] 
\draw [postaction={decorate}]  (0,0)-- node [left] {$t$} (0,2) ;
\draw [postaction={decorate}]  (0,2)-- node [above] {$a^p$} (3,2);
\draw [postaction={decorate}]  (3,2) -- node [right] {$t$} (3,0);
\draw [postaction={decorate}]  (0,0)-- node [below] {$a^q$} (3,0);
\end{scope}
\end{tikzpicture}} that fit together to form ``sheets" such as $\{t^n a^i : n,i\in \Z\}$, quasi-isometric to the hyperbolic plane $\H$.
The $t$-weight map $\tau: BS(p,q)\to \Z$ is a homomorphism given by
the total exponent of $t$ in a word, which is independent of
spelling. Its kernel is made up of countably many translates of $\langle a \rangle$.
 Consider one of  these translates $g\langle a \rangle$ and an infinite positive ray
  $r$ from this coset. By this, we mean a path $r$ starting in
  $g\langle a \rangle$ along which $\tau$ is monotone up and
  unbounded.  This ray picks out a {\em half-sheet} in the Cayley graph and
  this half-sheet is quasi-isometric to a horoball in the hyperbolic
  plane. For this reason $\langle a \rangle$ is called the {\em
    horocyclic subgroup} and its translates are seen as horocycles.
Note that Cayley graph can then be built by gluing together countably
many half-sheets.

Baumslag-Solitar groups are a frequent source of examples for group
theoretic properties.  $BS(p,q)$ is solvable (but not nilpotent) iff
$1= p < q$, and $BS(p,q)$ is automatic iff $p = q$.  In the solvable
case, the geometry of the horocyclic subgroup gives rise to a normal
form like that of decimal expansion.  For example, taking $BS(1,10)$,
we have that $t^2 a^5 T a^7 T a^3$ is a spelling of $a^{573}$ that
follows such a normal form.  Indeed, we can see the decimal expression
573 as a shorthand for this normal form with place value notation
recording conjugation by $t$ and numerals such as 7 standing in for
the coset representative $a^7$ of the coset $a^7t\langle a \rangle t^{-1}$ in $\langle a \rangle$.  Studying the growth of $\langle a\rangle\le BS(1,10)$ is close
in spirit to counting how many integers we can enumerate with each
number of digits.

The cases known to have rational growth with standard generators are
the solvable case \cite{BR,CEG} and the automatic case \cite{EJ}.  For
other $BS(p,q)$, both computing the growth series and deciding whether
it is rational are open problems.  One tool for studying their growth
is to study the subgroup growth of the horocyclic subgroup, and then use
this to draw conclusions about the growth of the whole group.
Restricting to the horocycle makes finding geodesics much more
manageable, as there exists an algorithm to produce unique geodesic
representatives for elements in the horocyclic subgroup due to McCann
(for $BS(2,3)$) and Schofield (modified for general $p$ and $q$)
\cite{FKS-GBSG1}.  These McCann-Schofield geodesics always have the
form $t^n a^k s$ where $s$ is a word in $T$, $a$ and $A$, where $t^n$
is called the {\em vertical prefix}, $a^k$ is called the {\em
  horizontal cap}, and $s$ is called the {\em suffix}.  In both the
above cases, the subgroup growth can been computed.  Using these
techniques, Freden, Knudson, and Schofield find that the subgroup has
rational growth when $p \mid q$ in \cite{FKS-GBSG1}, and conjecture
that the subgroup growth in $BS(2,3)$ is irrational.

\subsection{Horospheres}

Our groups (like Baumslag-Solitar groups) come equipped with a
canonical map $\tau : G \to \Z$ which yields the total $t$ weight of a
word, which does not depend on spelling.  For example, $\tau(ta^2
Ta)=0$.  This $\tau$, called the {\em height function},
 is the horofunction associated to the geodesic $\{t^n\}$.  
The kernel of the $\tau$ is a countable union of copies of our horocyclic
subgroup in both our groups and Baumslag-Solitar groups.  We can use
our $\tau$ to measure the maximum height of a spelling $g_1 \cdots g_{n}$
by taking the maximum $\tau$ of successive subwords $g_1 \cdots g_k$.
For example, $tataTaTa$ has height $0$ and maximum height $2$.

\subsection{Formal Languages and Finite State Automata} 
An {\em alphabet} is a finite set.  We refer to its elements as {\em letters} or {\em symbols}. Given 
an alphabet $\mathcal{A}$ we will denote by $\mathcal{A}^*$ the {\em free monoid} on $\mathcal{A}$, 
which the set of all finite strings formed by concatenations of letters in $\mathcal{A}$.
From this we call a subset $\mathcal{L} \subset \mathcal{A}^*$ a {\em formal language} over the alphabet $\mathcal{A}$.
One approach we will take is to build formal languages in bijective correspondance with geodesic normal forms of words in our groups.

Finally, we define a {\em finite state automaton} (FSA) to be an
abstract `machine' consisting of a finite number of states with
transitions labeled by symbols from our alphabet (which in our cases,
are often elements of our group), where some states are denoted as
{\em accept states}, and one state is designated as the {\em start
  state}.  It is convenient to specify an FSA by a labeled directed
graph, where states are vertices drawn as small circles and accept
states are vertices displayed as double circles.  We will adopt the
convention that the start state is labeled by $S$ or indicated by 
the word `start' on the diagram.
For example, we can build an FSA that accepts words in
our language of the positive quadrant of $\Z^2$ by

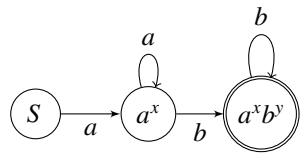
\begin{center}
\begin{tikzpicture}[->]
\tikzset{vertex/.style = {shape=circle,draw,minimum size=1.5em}}
\tikzset{edge/.style = {->,> = latex'}}

\node[vertex] (s) at  (0,0) {$S$};
\node[vertex] (a) at  (1.5,0) {$a^x$};
\node[vertex,double] (ab) at  (3,0) {$a^x b^y$};
\draw[edge] (s)  to [below] node {$a$} (a);
\draw[edge] (a)  to [below] node {$b$} (ab);

\draw (a) to [loop above]  node {$a$} (a);
\draw (ab) to [loop above]  node {$b$} (ab);

\end{tikzpicture}
\end{center}

We can compute the growth series of the language of an FSA
by taking the geometric series $ \sum_{i \geq 0} v_s^T Ax^i v_a =
v_s^T (I - Ax)^{-1} v_a$, where $A$ is the adjancecy matrix of the FSA
as a directed graph, $v_s$ is the vector with $1$ in the position
corresponding to the start state and zeros elsewhere, and $v_a$ is the
vector with $1$ in the rows corresponding to accept states and zeros
elsewhere.  This is a rational function of $x$ because the entries of a matrix $M^{-1}$ 
are rational functions of the entries of $M$.  If the 
FSA produces a unique geodesic representative for each element of a
group, we can use this method to compute the growth series of the
group.  

Going back to our example, we have precisely that 
$A = \left[
\begin{smallmatrix}
 0 & 1 & 0 \\
 0 & 1 & 1 \\
 0 & 0 & 1 \\
\end{smallmatrix}
\right],$  so that our growth series is 

$$ \left[
\begin{array}{ccc}
1 & 0 & 0 
\end{array}
\right]
\left[
\begin{array}{ccc}
 1 & \frac{x}{1-x} & \frac{x^2}{(1-x)^2} \\
 0 & \frac{1}{1-x} & \frac{x}{(1-x)^2} \\
 0 & 0 & \frac{1}{1-x} \\
\end{array}
\right]
\left[
\begin{array}{c}
0 \\
0\\
1
\end{array}
\right] = \frac{x^2}{(1-x)^2} = x^2 + 2x^3 + 3x^4 + 4x^5 + 5x^6 +  \ldots.$$

\subsection{Generalized Finite State Automata}

We have described FSAs using directed labeled graphs and have
  taken those labels to lie in our alphabet.  It is often convenient
  to take sets of spellings as the labels.  The resulting machines are
  sometimes referred to as {\em generalized FSAs}.  When the edges are
  labeled with regular languages, the resulting language is still
  regular.  One can see this by surgering in FSAs in place of the
  labeled edges.  We will only need the case where the sets are
  finite.  

More formally, a {\em generalized finite state automaton} (gFSA) is a tuple $M = (G,S,Y)$.
Here $G = (V,E,\phi)$ is a labelled directed finite graph where the
labelling $\phi$ assigns to each $e\in E$, a language $L_e$.  $S \in
V$ is the {\em start state}. $Y\subset V$ is the set of {\em accept
  states}.  For each pair of vertices $p,q \in V$, the language
$L_{pq}(M)$ is the set of words $w = u_1 \dots u_k$ where $e_{i_1}
\dots e_{i_k}$ is a path starting at $p$ and ending at $q$, and for
each $1 \le i \le k$, $u_i \in L_{e_i}$.  The {\em language} of $M$ is  
$ L(M) = \cup_{q \in Y} L_{Sq}(M).$
We say that $M$ has the {\em
  unique decomposition property} if for each $w \in L_{pq}(M)$ there
is a unique decomposition $w = u_1 \dots u_k$ with each $u_i$ in the
appropriate $L_e$. We call $k$, the {\em path length} of $w$ and take
$L^k_{pq}(M)$ to be the words of path length $k$ in $L_{pq}(M)$.

Just as we did with FSAs, we can talk about the growth of gFSAs.
For each edge $e\in E$, we let 
$ f_e(x) = \sum_{i=0}^\infty a_{ei} x^i $ 
be the growth of the language $L_e$.  (Here $a_{ei}$ is the number of words
in $L_e$ of length $i$.)
The {\em generalized adjacency matrix} of $M$ is
$$ A_{pq} =  \begin{cases}
f_e(x) &\text{if $(p,q)=e\in E$} \\
0 & \text{otherwise.} \end{cases}$$

\begin{prop}
Let $M$ be a gFSA with the unique decomposition property.
Then for all $p$, $q$, and $k$,  $(A^k)_{pq}$ gives the growth of $L^k_{pq}(M)$.
\end{prop}

\begin{proof}
The proof is by induction on $k$.  The basis step consists of
observing that for $k=0$, there is exactly one word, namely, the empty
word, in $L^0_{pp}(M)$, while for $p\ne q$, $L^0_{pq} = \emptyset$.
The induction step uses the following fact: Suppose that languages
$L_1$ and $L_2$ have growth functions $f_1(x)$ and $f_2(x)$.  Suppose
also that for each $w \in L_1L_2$ the decomposition $w=u_1 u_2$ is
unique. Then the growth of $L=L_1L_2$ is $f_1(x) f_2(x)$.

Now notice that $L^k_{pq}$ consists of the union of all
$L^{k-1}_{pr}L_{rq}$.  Since $M$ has the unique decomposition
property, for each $r$ the product language has the product growth.
The unique decomposition property also implies that this union is
disjoint.  The result now follows.
\end{proof}
An immediate corollary enables us to compute the growth in the same manner as we did for FSAs. 
\begin{cor}
The growth $f_M(x)$ of $L(M)$ is given by 
$$f_M(x) = \sum_{k=0}^\infty v_s  A^k v_a = v_s (I-A)^{-1}v_a,$$
where $v_s$ is the row vector with 1 as the entry for the
start state and 0s elsewhere and $v_a$ is the column
vector with 1 as the entry for states in $Y$ and 0s elsewhere.\qed
\end{cor}

We again compute the growth  of the language of the positive quadrant of $\Z^2$ ,
 but now by using the  gFSA with just a start state $S$ and an accept state $A$ with
edge languages $L_{SS} = L_{AS} = \emptyset$, $L_{SA} = a\{a\}^*b$, and $L_{AA} = b$.

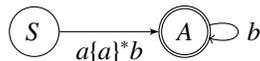
\begin{center}
\begin{tikzpicture}[->]
\tikzset{vertex/.style = {shape=circle,draw,minimum size=1.5em}}
\tikzset{edge/.style = {->,> = latex'}}
\node[vertex] (s) at  (0,0) {$S$};
\node[vertex, double] (a) at  (2,0) {$A$};
\draw[edge] (s)  to [below] node {$a\{ a \}^* b$} (a);
\draw (a) to [loop right]  node {$b$} (a);
\end{tikzpicture}
\end{center}

This clearly has the unique decomposition property, 
and we have the generalized adjacency matrix 
$A = \left[ \begin{array}{cc}
 0 & \frac{x^2}{1-x} \\
 0 & x 
\end{array} \right]$.  Thus our growth series is  

$$ \left[
\begin{array}{cc}
1 & 0 
\end{array}
\right]
\left[ \begin{array}{ccc}
 1 & \frac{x^2}{(1-x)^2} \\
 0 & \frac{1}{1-x} 
\end{array} \right]
\left[
\begin{array}{c}
0 \\
1
\end{array}
\right] = \frac{x^2}{(1-x)^2} = x^2 + 2x^3 + 3x^4 + 4x^5 + 5x^6 +  \ldots.$$
which matches exactly as it did when working with both a CFG and a FSA, but has a much cleaner derivation.

\subsection{Previous results for BS(1,3)} \label{sec-BS13}
This subsection recaps some results of Freden, Knudson, and Schofield in
\cite{FKS-GBSG1} (discussed in \cite{FredenOHGGT}) that our methods generalize.
We start by considering the horocyclic subgroup $\Z$
in BS(1,3). 

\begin{lem} \label{L1}
Let $U=\{ a, a^2, a^3, a^4\}$,  $V = \{ a^2, a^3, a^4 \}$, and $W =
\{ A, \epsilon, a \}$.  Let $L_{1,0} = U$ and for $n>0$ let  $L_{1,n} = t^n V (TW)^n$. 
Then the language $L = L_1 = \bigcup_{n \geq 0} L_{1,n}$ gives a unique geodesic for each $a^i$, $i>0$.
\end{lem}

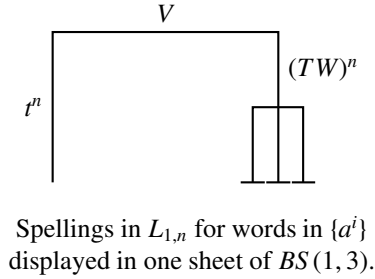
\begin{minipage}{3in}
\begin{center}
\begin{tikzpicture}
\draw[thick] (0,0) -- (0,1) node[left]{$t^n$} -- (0,2)  -- (1.5,2) node[above]{$V$} -- (3,2) -- (3,1.5) node[right]{$(TW)^n$} -- (3,1);
\draw[thick] (3,1) -- (2.66,1) -- (2.66,0);
\draw[thick] (3,1) -- (3,1) -- (3,0);
\draw[thick] (3,1) -- (3.33,1) -- (3.33,0);
\draw[thick] (2.84,0) -- (3.16,0);
\draw[thick] (2.5,0) -- (2.82,0);
\draw[thick] (3.18,0) -- (3.49,0);
\end{tikzpicture}

\bigskip

Spellings in $L_{1,n}$ for words in $\{a^i\}$\\ displayed in one sheet of $BS(1,3)$.

\end{center}
\end{minipage}
\quad 
\begin{minipage}{3in}
\[ \begin{array}{clr}
a^i & \text{Spelling}  & \ell_{\calG} (a^i) \\  \hline \hline
a & a & 1 \\ 
a^2 & a^2 & 2 \\
a^3 & a^3 & 3 \\
a^4 & a^4 & 4 \\ 
a^5 & ta^2TA & 5 \\
a^6 & ta^2T & 4 \\
a^7 & ta^2Ta & 5 \\
a^8 & ta^3TA & 6 \\
a^9 & ta^3T & 5 \\
a^{10} & ta^3Ta & 6 \\
a^{11} & ta^4TA & 7 \\
a^{12} & ta^4T & 6 \\
a^{13} & ta^4Ta & 7 \\ 
\vdots & \vdots & \vdots 
\end{array} \]
\end{minipage}

These are precisely the McCann-Schofield geodesics (sometimes called
`mesa paths'), and each $L_{1,n}$ spells every word $a^i$ for $i \in
\left[\frac{3^{n+1} + 1 }{2}, \frac{3^{n+2} -1 }{2} \right]$.
Freden \textit et al. \cite{FKS-GBSG1} exhibit a context free 
grammar for this language of geodesics. (For background on
  context free grammars and their automata see \cite{HU}.)  From this
context free grammar, the growth series of $\Z_+$ can computed by the
Delest-Sch\"utzenberger-Viennot (DSV) method \cite{CS} to be $P_1(x) =
\frac{x +x^2 - x^4}{1 - x^2 - 2x^3} $.  This makes the growth of $\Z$
equal to $1 + 2 P_1(x)$, which implies

\begin{prop}[\cite{FKS-GBSG1}]
The subgroup growth series of $\Z$ in $BS(1,3)$  is rational.
\end{prop}

Finally, using the a branching recursion, they conclude 

\begin{thm}[\cite{CEG, FKS-GBSG1}]
The growth series of $BS(1,3)$  is rational.
\end{thm}

We will similarly exhibit a language of geodesics for the horocyclic subgroups of $\HNN$
for arbitrary $m$, but will compute growth without the use of a context free grammar.
The idea for $BS(1,3)$ is as follows:

\begin{remark}\label{rem-bijectsToRegular}
Consider the language
\[
R = \{t^n (TW)^n \mid n > 0 \}.
\]
Clearly $L$ is in length preserving bijection with the language $U
\cup V R$.  Using $U(x)$, $V(x)$ and $R(x)$ to denote the growth of
$U$, $V$ and $R$, it follows that the growth of $\Z_+$ is $P_1(x) = U(x) + V(x) R(x)$.  Now $R$
is in length-preserving bijection with 
\[
R' = \{(tTW)^n \mid n > 0 \}.
\]
and this latter is regular.  Thus, if we are willing
to pass to a language which no longer gives geodesics in the group, we
can count the subgroup growth of $\Z \subset BS(1,3)$ using generalized finite
state automata.
\end{remark}

This method is precisely how we will show rationality for all $m$.
To start, though, we first need to describe a language for the horocyclic subgroup at each $m$.

\section{A Language of Geodesics for the $\Z^m$ Subgroup}\label{sec-lang}

In this section we will we consider $\Z^m < \Z^m *_{g\mapsto
  g^3}$. We will extend the methods of the previous section to produce
a language $L_m$ giving geodesics for the positive orthant.  The spellings in
our language consist of a power of $t$ followed by a cap which is
followed by a suffix.

\subsection{The families of sets $W_m$, $U_m$, and $V_m$} 
To define our language, we define the various sets necessary to build the caps and suffixes.

\begin{defn}
For each $m$,  let $W_m$ be the concatenation $ W_m := \{a_1,\epsilon,A_1\}  \cdots \{a_m,\epsilon,A_m\}$.
Equivalently,  we can write this as $W_m = \{ a_1^{e_1} \dots a_m^{e_m} \mid e_i \in \{-1,0,1\} \}$.
\end{defn}

Notice that there is the relationship $W_m = W_{m-1} \{ \epsilon, a_m, A_m \}$.
We will see that our suffixes will be elements of $(TW_m)^*$,
and our prefixes will be the corresponding elements of $\{ t \}^*$ such that the total height $\tau$ is $0$.

\begin{defn}
For each $m$, $U_m := \{a_1 \cdots a_m  \}$ is the set consisting of the shortest word in the positive  orthant.
\end{defn}

The word $a_1 \cdots a_m$ has length $m$ and is best spelled with no conjugations.
It is not the only one to be best spelled with only $a_i$, but it will be notationally convenient to single it out,
particularly when building the set of caps.

\begin{defn}
For each $m$, $V_m := \left( U_m^2 \bigcup (tU_m TW_m) \right) \setminus (tU_mTU_m^{-1})$
 is set consisting of top dimensional caps.
\end{defn}
We can think of $V_m$ as the set consisting of the all squares word from $U_m^2 =\{ a_1^2 \cdots a_m^2 \}$,
along with all conjugations of the word in $U_m$, except the one evaluating to the all squares word $a_1^2 \cdots a_m^2$.
This set is top dimensional in the sense that all $m$ letters from $\Z^m$ are present in it.
The important thing to notice is that $V_m$ is a set of geodesic spellings of words with exponents either $2$, $3$ or $4$.
For example, $ V_2 = \{a^2b^2, tabT, tabTA, tabTB, tabTa, tabTb, tabTab, tabTAb, tabTaB \}$.

Unlike the language we built previously for the $BS(1,3)$ case, our one dimensional caps 
are $V_1 = \{ a^2, taT, taTa \}$ and  not $\{a^2, a^3, a^4 \}$.
The reason for fussing with conjugations and the all squares words rather than just taking concatenations of powers of $2$, $3$, and $4$
is that for $m > 1$ only the all squares word remains a geodesic spelling of itself.  
For example, $a^3 b^3$ is not a geodesic because $tabT$ spells it in fewer letters.

These three families of sets will be sufficient to build our language, 
but we must now concern ourselves with how we piece them together.
In particular, while it is true that top dimensional caps from $V_m$ will be caps for spellings in $L_m$, 
not all caps used by $L_m$ will be top dimensional.

\subsection{Defining the language $L_m$} 
First we define the language $L_1$ of $\Z$ in $BS(1,3)$ to be 
$L_1 = U_1 \bigcup_{n \geq 0} L_{1,n}$, where 
$U_1 = \{ a \}$, and for each $n\geq 0$,  
$L_{1,n} = t^n V_1 (TW_1)^n$, 
with  $V_1 = \{ a^2, taT, taTa \}$ and $W = \{ A, \epsilon, a \}$ as defined in the previous subsection.
This is a modified version of the language constructed in Lemma \ref{L1},
where now we are taking caps $\{a^2, taT, taTa \}$ instead of $\{a^2, a^3, a^4\}$ . 
The lemma implies immediately that this language $L_1$ spells every word in $\Z_+$ uniquely and geodesically.

For $L_1$, we only needed caps from $V_1$. 
For higher $m$, there is a problem in that if we just conjugate caps from $V_m$, we couldn't possibly spell words 
that project to words spelled by different $L_{1,n}$s in their respective $BS(1,3)$ subgroups.
For example, consider $a^5 b^4$: the element $a^5$ 
is best spelled as $ta^2TA \in L_{1,1}$, but $b^4$ is best spelled $tbTb \in L_{1,0}$.
The smallest cap we have is in $V_2$ is $a^2 b^2$, which conjugates to $a^6 b^6$.
Since elements of $W_2$ can only decrease the powers by at most one, we cannot spell $a^5 b^4$.
To geodesically spell $a^5 b^4$, we would like to combine the two spellings to get $ta^2 b TAb$,
which should be thought of as taking the cap $a^2 \in V_1 $ and adding on the $U_1$ word $b$, 
then conjugating via the suffix $TAb \in TW_2$. 
This demonstrates that we do not just need top dimensional caps, but all possible lower dimensional caps 
concatenated with $U$ sets as well.
We will need to use partitions and permutations to properly utilize lower dimensional caps.

\begin{defn}
For a pair of positive integers $(m,n)$, 
denote by $(I, J)$ a pair of ordered partitions 
$I = (i_1, \ldots, i_q)$ and $J = (j_1, \ldots , j_q)$
where $I$ is a partition of $m$ into positive integers ($ i_1 + \cdots + i_q = m$),
and $J$  is a partition of the same length of $n$ into non-negative integers ($j_1 + j_2 + \ldots + j_q = n$).
Furthermore, denote by $\Delta i_k$ the range $\{i_{k-1}+1, \ldots, i_{k-1}+i_k \}$, 
with the convention $i_0 = 0$.
\end{defn}

We can think of $(I,J)$ as a partition of $(m,n)$ into an ordered set
of pairs $(i_k, j_k)$ such that the $i$ sum to $m$ and the $j$ sum to
$n$.  The $(I,J)$ partitions will help keep track of how many letters
have been seen and at what height.  We now have to worry not just how
many letters are present in a set, however, but {\em which} letters
are present in a set.  To do so, we define a new subscript as below:

\begin{defn}
Denote by $\Sigma_m$ the group of permutations of our $m$ generators. 
For a permutation $\phi : \{a_1, \ldots, a_m \} \to \{a_1, \ldots, a_m \}$,
Let $X \subset \{a_1, \ldots, a_m \}$ so that $\phi(X)$ is its image,
and denote $W_{\phi(X)}$, $U_{\phi(X)}$, and $V_{\phi(X)}$
to be the images of the set $W_{|X|}$, $U_{|X|}$, $V_{|X|}$
 under the permutation.

\end{defn}

This now gives us a way to keep track of what letters are contained in our sets.
For example, in the $m=2$ case, taking $\phi = Id$ and $X = \{a\}$, 
we can speak of $U_a = \{ a\}$ as the image of $U_1$, but taking $X = \{a, b\}$, we have $U_{ab} = \{ab\}$ as the image of $U_2$.
On the other hand if $\phi$ permutes $a$ and $b$, with $X = \{a \}$, we have $U_{b} = \{ b \}$.

\begin{defn} \label{def:language}
For $n \geq 0$,  let

$$ L_{m,n} =    \bigcup_{ \phi \in \Sigma_m} \ \bigcup_{\text{Partitions } (I,J) \text{ of } (m,n) }
t^n  V_{\phi( \Delta i_1) } (TW_{\phi(\Delta i_1)}   )^{j_1} 
U_{\phi( \Delta i_2)} (TW_{\phi(\Delta i_1) \cup \phi(\Delta i_2) }   )^{j_2}
  \cdots U_{\phi( \Delta i_q)} (TW_{m}   )^{j_q}. $$

and finally define, define $L_m = U_m \bigcup_{n \geq 0} L_{m,n}$.  
\end{defn}

There is a concise description of $L_m$ as the
  intersection of a regular language with the requirement that $\tau =
  0$.  We define a generalized finite state automaton $M$ as follows.
  The states of $M$ are $\calP \cup \{ U\}$ where $\calP$ denotes
  the power set of $\{1,\dots,m\}$, and $U$ is a special state corresponding to $U_m$. 
  The start state is $\emptyset$. There are two accept states,
  $\{1,\dots,m\}$ and $U$. There is an edge from $\emptyset$ to $U$ labeled 
  by the singleton set $U_m$, and a loop at $\emptyset$ 
  labeled by $t$.  For each $X \in \calP \setminus \emptyset$, there
  is an edge from $\emptyset$ to $X$ labeled by $V_X$.  For
  each $X \neq \emptyset$, there is a loop at $X$ labeled by $TW_X$.
  For each $\emptyset \neq X \subsetneq X'$, there is an edge from
  $X$ to $X'$ labeled by $U_{X'\setminus X}$.  For a word
  labeling a path in this generalized FSA, the initial $V$ and
  subsequent $U$s determine the partition of $m$ and the permutation
  $\phi$ in the definition of $L_m$. The number of times that word
  traverses each $V_X$ edge and loop determines $n$ and the partition
  of $n$.  The requirement that $\tau = 0$ ensures that this is
  preceded by $t^n$. 
The gFSA in the $m=2$ case is illustrated in Figure \ref{fig:Z2Machine}.

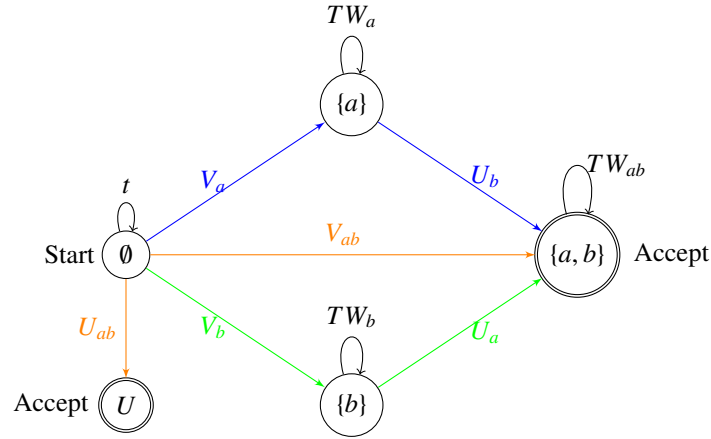
\begin{figure}[ht]
\begin{tikzpicture}[->]
\tikzset{vertex/.style = {shape=circle,draw,minimum size=1.5em}}
\tikzset{edge/.style = {->,> = latex'}}


\node[vertex, double] (U) at  (0,-2) {$U$};

\node[vertex] (empt) at  (0,0) {$\emptyset$};
\node[vertex] (a1) at  (3,2) {$\{a\}$};
\node[vertex] (b1) at  (3,-2) {$\{b\}$};
\node[vertex, double] (ab1) at  (6,0) {$\{a, b\}$};
\node (start) at  (-.75,0) {Start};
\node (accept1) at  (-1,-2) {Accept};
\node (accept2) at  (7.25, 0) {Accept};

\draw[edge, orange] (empt)  to [left] node {$U_{ab}$} (U);

\draw (empt) to [loop above]  node {$t$} (empt);
\draw (a1) to [loop above]  node {$TW_a$} (a1);
\draw (b1) to [loop above]  node {$TW_b$} (b1);
\draw (ab1) to [loop above,right]  node {$TW_{ab}$} (ab1);

\draw[edge,blue] (empt)  to [left] node {$V_a$} (a1);
\draw[edge,green] (empt)  to [left] node {$V_b$} (b1);
\draw[edge, orange] (empt)  to [above] node {$V_{ab}$} (ab1);

\draw[edge,blue] (a1)  to [right] node {$U_b$} (ab1);
\draw[edge,green] (b1)  to [right] node {$U_a$} (ab1);
\end{tikzpicture}
\caption{The intersection of the language produced by this machine with $\tau = 0$ yields $L_2$,
where  
{\color{green}  $U_a = \{a \}$},
{\color{blue} $U_b = \{b\}$ },
{\color{orange} $U_{ab} = \{ab \}$}, 
{\color{blue} $V_a = \{ a^2, taT, taTa \}$}, 
{\color{green} $V_b = \{ b^2, tbT, tbTb \}$}, 
{\color{orange} $V_{ab} = \{a^2b^2, tabT, tabTA, tabTB, tabTa, tabTb, tabTab, tabTAb, tabTaB \}$},
$W_a = \{ \epsilon, a, A \}$,
$W_b = \{ \epsilon, b, B\}$, and
$W_{ab} = \{ \epsilon, a, A, b, B, ab, Ab aB, AB\} $ } 
\label{fig:Z2Machine}
\end{figure}

\begin{prop}
 $L_{m,n}$ spells words in 
$\left[0, \frac{3^{n+2} -1 }{2} \right]^m \setminus \left[0,\frac{3^{n+1} + 1 }{2} \right)^m \subset \Z_+^m$.   Spellings in $L_{m,n}$ have a max height of either $n$ or $n+1$.
\end{prop}

\begin{proof}
The statement about the height of the spellings follows from the fact that the height of the caps varies.
The $m = 1$ case is readily observed from the fact that projecting to one coordinate is equivalent to the language in Lemma \ref{L1}. 
From the $m = 1$ case, one can observe that the largest coordinate must be a member of $L_{1,n}$ and thus must 
evaluate to an element of $\left[\frac{3^{n+1} + 1 }{2}, \frac{3^{n+2} -1 }{2} \right]^m$.
All other coordinates can be in any other level of $L_{1, i}$, and can evaluate to any thing in  $\left[0, \frac{3^{n+2} -1 }{2} \right]^m$.
\end{proof}

When counting, it will be useful to break our language up into the subsets 
that $L_{m,n}$ is made up of and then group them by which letters they have via the permutations.
For example, in Section \ref{sec-Z2}, we will break $L_2$ up into a union of four disjoint sets, 
$U_{ab}$, $\bigcup t^n V_{ab} (TW_{ab})$,
$\bigcup t^n V_a (TW_a)^\alpha U_b (TW_{ab})^\beta$, 
and $\bigcup t^n V_b (TW_b)^\alpha U_a (TW_{ab})^\beta$ 
as seen in Figure \ref{fig:Z2Words} at the end of section \ref{sec-Z2}.

It is easy to get lost in the setup, but what the words in the
language look like is actually quite simple if you think of them as
base-3 expansions where the highest digit must be positive and every
other digit can be $-1$, $0$, or $1$ with the added caveat that the
highest digit can be $2$ only if it has the highest place value among
all $m$ entries.

To illustrate this, consider the $m=2$ case, consider the group
element $a^{10} b^{16} \in \Z^2_+$.  In $L_2$, it is spelled $t^2 ab^2
TaB Tb$ since $10 = 1\cdot 3^2 + 0\cdot 3 + 1\cdot 1$ and $16 = 2\cdot
3^2 - 1\cdot 3 + 1\cdot 3$.  Here the highest place value is $2$ since
at least one coordinate has their base 3 expansion in this way with a
term on $3^2$.  On the other hand, $a^{2} b^{16} = t^2 ab^2TBTab$  
since $2 = 1\cdot 3 - 1\cdot 1$ leads with $1$. 
($2\cdot 1$ was not chosen since it leads with 2 but not as the $3^2$
term). 

\begin{thm}\label{LanguageTheorem}
$L_m$ is a language of geodesics bijecting to the positive orthant $\Z_+^m \subset \Z^m *_{g\mapsto g^3}$. 
\end{thm}

\begin{proof}
Tracking through the setup yields that $L_m$ spells all words in the
positive orthant uniquely as they are in bijection with their base-3
expansions.  What remains to be proven is that they are geodesic.  We
do this inductively, with $m=1$ being our base case, in which, by the
observation above, our spellings are equivalent to the
McCann-Schofield geodesics, which we noted were geodesic in \S
\ref{sec-BS13} (it can also be shown through an induction on the
height of the geodesic similar to as we do when inducting on $m$).
Now, assume that for all $k < m$, $L_k$ is a language of geodesics
bijecting to the positive orthant of $\Z^k \subset \Z^k *_{g\mapsto
  g^3}$.  The inductive step is that this implies that $L_m$ is a
language of geodesics for $\Z_+^m \subset \Z^m *_{g\mapsto g^3}$.

In order to prove this, we prove that all geodesic spellings of
elements in the positive orthant are equivalent to a spelling from
$L_m$ by inducting on the maximum height (this is an induction to
prove the inductive step of the theorem).  For the base case, consider
geodesic spellings that use no conjugations.  Obviously the two
spellings $a_1 \cdots a_m$ and $ a_1^2 \cdots a_m^2 \in L_m$ are the
only geodesic spellings of those words, and $a_1^3 \cdots a_m^3$ is
not a geodesic spelling since it can also be expressed as $ta_1 \cdots
a_m T$ in less letters.  Then the only other possible geodesic
spellings are of words in $[1,3]^m$, which there are finitely many.
In the case that the highest power is $2$, then it is clearly in
$L_{m,0}$ as an element of $V'_m$.  ex. $a_1^1 a_2^2$ is geodesic and
in $L_{3,0}$.  In the case that the highest power is $3$, then it
can be expressed as an element of $L_{m,0}$ by replacing all the
powers of $3$ with a conjugation $taT$, replacing all powers of two by
a conjugation like $taTA$, and forcing all conjugations to sharing the
same $t$ pair.  ex. $a_1^1 a_2^2 a_3^3 = t a_2 a_3 T a_1 A_2 $ is
geodesic and in $L_{3,0}$.

Now, assume that all spellings of max height $< n$ are equivalent to a
unique spelling in our language, and consider a geodesic spelling $u$
of max height $n$.  We first prove that $u$ is equivalent to a
geodesic spelling starting with $t^n$ with $\tau$ non-increasing over
successive subwords, and then show that such a geodesic spelling is
equivalent to a spelling our language.  Start by noting that if $u 
= u_0 T u_1 t u_2$ were a geodesic spelling, then $u_1$ must evaluate
to a word in $\langle a_1^3, \ldots, a_m^3 \rangle$, and we can notice
that such words can be geodesically spelled as conjugations, which
would then force a $Tt$ cancellation, contradicting geodicity.  Thus
geodesic spellings must have a $t$ come before a $T$.  Find the first
instance of instance of a $t$ in $s$ and the last instance of $T$ in
$s$ in our spelling $u$, and split our word as $u = u_0 t u_1 T u_2$,
where $u_0$ and $u_2$ have no powers of $t$.  Both $u_0$ and $t u_1 T
u_2$ are equal to words in $\Z^m$, and thus commute, so $u = t
u_1 T u_2 u_3$.  We must investigate what $u_1$ can be.

Case I: $u_1$ is a geodesic spelling  of an element in the positive orthant of the horocyclic subgroup.  
Then it has max height $n-1$, so by the inductive hypothesis, 
$u_1$ is equivalent to some $t^{n-1} v s \in L_{m}$.
This makes the spelling $t^n v s T u_2 u_0 =  u$ geodesic. 

If $u_2 u_0 \in W_m$, then we have that  $ t^n v s T u_2 u_0$ is geodesic and in our language, so we are done.
Else, we have to construct a new word in our language.
For this to happen, then for some $i$, the $a_i$ power of $u_2 u_0$ is not $-1$, $0$, or $1$.
If the $a_i$ power is greater than $2$, then if we use $Ta_i^3 = a_it$, we have shortened our word, which contradicts our spelling being geodesic, 
and similar for the power being less than $-2$.
If the $a_i$ power is $2$, then we can use that $Ta_i^2 = a_iTA_i$ to slide the power up one level and similar for the power of $-2$.
If the resulting spelling is in our language, then we are done.
If it isn't, then that means we have an $a_i^2$ at that level, and we can iterate the process.
At worst this terminates when we reach $v$, in which case if the power in $v$ is 0 or 1, we increase it to 2 and are done.
If the power in $v$ is $2$, then to deal with our power of $3$, we replace $v a_i$ with an element of $L_{m,1}$ that evaluates to $v' a_i$,
which gives us a new spelling in $L_{m,n+1}$.

Case II: $u_1$ is a geodesic spelling of an element in the horocyclic
subgroup, but not in the positive orthant.  In this case, the power of
$a_i$ on the evaluation of $u_1$ is $0$ for some indices $i$ (i.e. we
did not use a top dimensional cap).  If the power is nonzero for $k$
indices, then consider the isometrically embedded copy of $\Z^k*_{g
  \mapsto g}$ spanned by these indices plus $t$.  We can spell $u_1$
from $L_k$ geodesically as $t^{n-1} v s$.  This makes the spelling
$t^n v s T u_2 u_0$ geodesic, and, additionally, $u_2 u_0$ must
evaluate to non-zero powers for the other $m-k$ indices.  Now we run
the same argument as in Case I to finish this out.
\end{proof}

\begin{remark}
We have seen that $L_m$ gives a set of unique geodesics for the
positive orthant of $\Z^m$. We have also given a description of this
language using a generalized FSA and the requirement that $\tau =
0$. Using the methods sketched in Remark~\ref{rem-bijectsToRegular},
it now follows that the subgroup growth of the positive orthant is
rational.  However, the machine we described has $2^{m+1}$ states.
Further, since this is a generalised FSA with epsilon transitions, we cannot easily count words.
We now describe smaller machines which will allow us to perform explicit calculations.
These explicit calculations will be helpful in proving the whole group has rational growth.
\end{remark}

\section{The Prefix/Suffix Series $R_m(x)$} \label{sec-conj}

We return now to counting to create a series whose coefficients count prefix/suffix pairs. 
Just as noted in Remark \ref{rem-bijectsToRegular} for the $BS(1,3)$ case, 
for each $m$ there is a length preserving bijection 
between prefix/suffix pairs and the language  $R_m = \{t^n (TW_m)^n \mid n > 0 \}$, 
(where here $t$ and $T$ are symbols in an alphabet and not group elements)
which is itself in length preserving bijection with 
$R_m' = \{(tTW_m)^n \mid n > 0 \}$, which is regular (although it no longer accepts geodesics in the group).

An FSA to compute the growth of $R'_m$ would have $O(m^2)$ states 
but we can easily compute this as a gFSA with a single edge and a loop labeled by the language $tTW_m$.

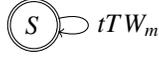
\begin{center}
\begin{tikzpicture}[->]
\tikzset{vertex/.style = {shape=circle,draw,minimum size=1.5em}}
\tikzset{edge/.style = {->,> = latex'}}
\node[vertex, double] (s) at  (0,0) {$S$};
\draw (s) to [loop right]  node {$tTW_m$} (s);
\end{tikzpicture}
\end{center}

\begin{prop}
The Prefix/Suffix Series is given by $R_m(x) =  \left( 1 -  \sum_{i=0}^m {m \choose i}  2^i x^{i+2}  \right)^{-1}$. 
\end{prop}
\begin{proof}
First observe that this gFSA has the unique decomposition property.

The adjacency matrix of the gFSA for the language $R'_m$ is just the $1 \times 1$ matrix consisting of $\{ x^2 W_m(x) \}$,
where $W_m(x)$ is the polynomial equal to the growth of $W_m$.
Thus the growth is just $\frac{1}{1 - x^2 W_m(x)}$
Recall that $ W_m$ is defined as $\{a_1,\epsilon,A_1\}  \cdots \{a_m,\epsilon,A_m\}$, so $W_m(x) = (1 + 2x)^m$, and our formula follows from the binomial expansion.
\end{proof}

\section{The Cap Polynomial $V_{m}(x)$} \label{sec-cap}
In the previous section we computed $R_m(x)$, which counts how many ways there are to contribute $n$ letters by conjugation by a prefix/suffix pair.
However, we did not count how many letters the thing we are conjugating contributes to the word in the first place. 
We seek to do so now.

Recall that we defined $V$ to be the finite set of things we conjugate (caps) uniquely spelled geodesically, 
and defined the {\em cap polynomial} to be the corresponding growth series, which is polynomial since $V$ is finite.
There is some freedom to choose our caps, as we can recall from the $BS(1,3)$ case, in that to geodesically spell a word in $\Z_+$, 
we have two equivalent sets of caps, $\{a^2, a^3, a^4 \}$ or $\{a^2, taT, taTa \}$.
For higher $m$, equivalent sets of caps are a bit harder to describe, but we can focus on 
the top dimensional caps, which are spellings of the group elements $a_1^{i_1} \ldots a_m^{i_m}$ with $2 \leq i_k \leq 4$.
Define the {\em cap polynomial for $\Z^m$ in $\Z^m *_{g \mapsto g^3}$ } to be $V_m(x) = \sum \eta_{m,n} x^n$, 
where $\eta_{m,n}$ is the number of words $a_1^{i_1} \ldots a_m^{i_m}$ with $2 \leq i_k \leq 4$ for all $k$ of word length $n$.
Computing this then amounts to finding the geodesic spellings of all such words, which we have already done in \S \ref{sec-lang} by the family of sets $V_m$. 
For example, these are the caps we use for $m=2$ case

\[ \begin{array}{clr}
a^i b^i & \text{Spelling from } V_2  & \ell_{\calG} (a^i b^j) \\ \hline \hline
a^2 b^2 & a^2 b^2 & 4 \\ 
a^2 b^3 & tabTA  & 5 \\ 
a^2 b^4 & tabTAb & 6 \\ 
a^3 b^2 & tabTB & 5 \\ 
a^3 b^3 & tabT & 4 \\ 
a^3 b^4 & tabTb & 5 \\ 
a^4 b^2 & tabTaB & 6 \\ 
a^4 b^3 & tabTa & 5 \\ 
a^4 b^4 & tabTab & 6 \\ 
\end{array} \]

We first can notice that $a^3$ and $a^4$ have geodesic spellings $taT$ and $taTa$, but that $a^2$ is not spelled geodesically by $taTA$.
Observe then from the geodesic spelling  $taT$ of $a^3$, by adding in two letters we can get a geodesic spellings of $a^3 b^2$ and $a^3 b^4$.
Similarly, by adding in just one letter, we can get $a^3 b^3$. 
Similar is true if we start with the geodesic spelling $taTa$ of $a^4$.
However, starting from the only geodesic spelling $a^2$ of $a^2$, we must add two letters to get to get geodesic spellings of $a^2 b^2$.
Equivalently, we can think of this as starting with the geodesic spelling $b^2$ of $b^2$, and adding two instances of $a$.
If we want to spell $a^2 b^3$ or $a^2 b^4$ the geodesically, we must start with geodesic spellings $b^3$ and $b^4$  and add two instances of $a$ there as well.

\begin{prop}
The polynomials $V_{m}(x)$ are related through the recursion:
$$ V_{m}(x) = (x + 2x^2)( V_{m-1} - x^{2(m-1)} ) + x^{2(m-1)}(V_{1}) \text{ , Where } V_{1}(x) = x^2 + x^3 + x^4$$
\end{prop}

\begin{proof}
We prove this by induction, with the base case being the $m=2$ case described above. Note:
$$(x + 2x^2)(x^3 + x^4) + x^2 (x^2 + x^3 + x^4) = 2x^4 + 4x^5 + 3x^6 .$$

The induction follows the same argument as the construction of the $m=2$ case above:
If we take any word $w = a_1^{i_1} \ldots a_{m-1}^{i_{m-1}}$ except $a_1^2 \cdots a_{m-1}^2$, 
it must have have a geodesic spelling with at least one conjugation since any $a_i^3$ can be spelled $ta_iT$. 
We can get to a geodesic spelling of $w a_m ^3$ by adding one $a_m$ before the last $T$.
Similarly, we can get geodesic spellings of $w a_m^2$ and $w a_m^4$ by adding an $a_m$ before the last $T$ and an $a_m$ or an $A_m$ after..  

The only remaining case is when you start with all the generators squared, which has length $2m - 2$.  
Should you do that, then to get a geodesic spelling of $a_1^2 \cdots a_{m-1}^2 a_m^2$ takes adding two more letters, 
$a_1^2 \cdots a_{m-1}^2 a_m^3$ takes three more letters, and $a_1^2 \cdots a_{m-1}^2 a_m^4$ takes four more letters.
\end{proof}

\begin{cor}\label{cap-formula}
The cap polynomial $V_m(x)$  is equal to $x^{2m} + x^{m+2}W_m(x) - x^{2m+2}$.
\end{cor}
\begin{proof}
The formula comes from the construction in \S \ref{sec-lang} of $V_m = \left( U_m^2 \bigcup (tU_m TW_m) \right) \setminus (tU_mTU_m^{-1})$.
However, the above proposition means that if we use an equivalent set of caps, the induction formula must hold,
so we prove that the formula satisfies the inductive formula from the previous proposition.
For $m = 1$ we get exactly $x^2 + x^3 + x^4$ as desired.
Now assume that $V_{m-1}(x)  = x^{2m-2} + x^{m+1}(1+2x)^{m-1} - x^{2m}$ as in the formula.

\begin{align*}
V_m(x) &= x(1+2x)(V_{m-1} - x^{2m-2}) + x^{2m-2}(x^2 + x^3 + x^4)\\
 &= x(1+2x)( x^{m+1}(1+2x)^{m-1} - x^{2m}  )           + x^{2m} + x^{2m+1} + x^{2m+2}\\
 &= x^{m+2}(1+2x)^m - x^{2m+1} - 2x^{2m+2} + x^{2m} + x^{2m+1} + x^{2m+2}  \\
 &= x^{m+2}(1+2x)^m  + x^{2m} - x^{2m+2} 
\end{align*}
Recalling that $W_m(x) = (1+2x)^m$ yields the formula as desired.
\end{proof}

\section{Subgroup Growth for $\Z^2$}\label{sec-Z2}
We would now like to move on to the $\Z^2$ case and see that it agrees with Freden's computation in \cite{FredenOHGGT}.
The group presentation we use is $\langle a, b, t \mid taT = a^3, tbT = b^3, [a,b] = 1 \rangle$.
From Theorem  \ref{LanguageTheorem}, we have our language of geodesics $L_2$.
Recalling our definition of $L_2$, we can break up our union by which caps we are using.
Equivalently, we are breaking up $L_2$ by whether whether the $a$ projection has higher max height,
the $b$ projection has higher max height, or if they have the same max heights.
\begin{align*}
L_2 = U_{ab} &\bigcup_{n \geq 0} t^n V_{ab} (T W_{ab} )^n & \\
  &\bigcup_{n \geq 0} t^n V_{a} (T W_{a} )^\alpha U_{b} (T W_{ab})^\beta & \alpha + \beta = n \\
   &\bigcup_{n \geq 0} t^n V_{b} (T W_{b} )^\alpha U_{a}  (T W_{ab})^\beta & \alpha + \beta = n 
\end{align*}
 
where $U_a = \{a \}$,
$U_b = \{b\}$,
$U_{ab} = \{ab \}$, 
$V_a = \{ a^2, taT, taTa \}$, 
$V_b = \{ b^2, tbT, tbTb \}$, \\
$V_{ab} = \{a^2b^2, tabT, tabTA, tabTB, tabTa, tabTb, tabTab, tabTAb, tabTaB \}$,
$W_a = \{ \epsilon, a, A \}$,
$W_b = \{ \epsilon, b, B\}$, and \\
$W_{ab} = \{ \epsilon, a, A, b, B, ab, Ab aB, AB\} $,
as in \S \ref{sec-lang}. 

The first union corresponds to going up $n$, using a cap spelled in $a$ and $b$, 
and then coming down in any path.
The other two unions corresponds to going up $n$, using a cap in one letter, 
coming down in any branching path in that letter, moving out in the other letter once at at height $n-\alpha$,
and then coming down in any path.
Figure \ref{fig:Z2Words} shows how each of these families sit in the positive quadrant.
Thus, the growth series of the positive quadrant is equal to 
$$P_2(x) =  U_{2}(x) + V_{2}(x)R_2(x) + 2 (V_{1}(x)  R_1(x) ) ( U_1(x)  R_2(x)) .  $$

Where $U_1(x) = x$ is the growth of $U_a$ and $U_b$,
$U_2(x) = x^2$ is the growth of $U_{ab}$,
 $R_1(x)$ and $R_2(x)$ are prefix/suffix series as in \S\ref{sec-conj},
 and $V_1(x)$ and $V_2(x)$ are the cap polynomials as in \S\ref{sec-cap}.

Recalling $P_1(x)$ as the growth of the positive axis computed from $BS(1,3)$ in \S\ref{sec-BS13},
 we can now compute our subgroup growth series.
Since $\Z^2$ has four quadrants, four axis, and the origin, then the subgroup growth series for $\Z^2$ in $G$ is
$$ 1+ 4P_1(x) + 4P_2(x) =  \frac{ (1 - x) (1 + 2 x + 2 x^2)^2}{(1 - 2 x) (1 + x + 2 x^2)} = 1 + 4 x + 8 x^2 + 12 x^3 + 24 x^4 + 52 x^5 + 100 x^6 + 196 x^7 + 
 404 x^8 + \ldots $$

\newpage

\begin{figure}[ht]
$$  L_2 = {\color{orange} U_{ab} } \bigcup_{n \geq 0}
\begin{cases}
{\color{orange} t^n V_{ab} (T W_{ab} )^n } &  \\
{\color{blue} t^n V_{a} (T W_{a} )^\alpha U_{b} (T W_{ab})^\beta} &  \alpha + \beta = n\\
{\color{green} t^n V_{b} (T W_{b} )^\alpha U_{a}  (T W_{ab})^\beta} &  \alpha + \beta = n
\end{cases} 
$$

 $U_a = \{a \}$,
$U_b = \{b\}$,
$U_{ab} = \{ab \}$,

 $V_a = \{ a^2, taT, taTa \}$,
 $V_b = \{ b^2, tbT, tbTb \}$, 

$V_{ab} = \{a^2b^2, tabT, tabTA, tabTB, tabTa, tabTb, tabTab, tabTAb, tabTaB \}$,

$W_a = \{ \epsilon, a, A \}$,
$W_b = \{ \epsilon, b, B\}$, and
$W_{ab} = \{ \epsilon, a, A, b, B, ab, Ab aB, AB\}$\\

\begin{tikzpicture}[scale=1.2]

\tikzset{axis/.style = {->,> = latex'}}
\tikzset{TopLeftFar/.style = {fill=green, opacity=0.5}}
\tikzset{TopLeft/.style = {fill=green, opacity=0.5}}
\tikzset{Middle/.style = {fill=orange, opacity=0.5}}
\tikzset{BottomRight/.style = {fill=blue, opacity=0.5}}
\tikzset{BottomRightFar/.style = {fill=blue, opacity=0.5}}
\def\x{0.1}

\draw[axis] (0,0) to (13,0);
\draw[axis] (0,0) to (0,13);

\node at (0.5,-1) {$a^1$};
\node at (1.5,-1) {$a^2$};
\node at (2.5,-1) {$a^3$};
\node at (3.5,-1) {$a^4$};
\node at (4.5,-1) {$a^5$};
\node at (5.5,-1) {$\ldots$};
\node at (6.5,-1) {$a^{13}$};
\node at (7.5,-1) {$a^{14}$};
\node at (8.5,-1) {$\ldots$};
\node at (9.5,-1) {$a^{40}$};
\node at (10.5,-1) {$a^{41}$};
\node at (11.5,-1) {$\ldots$};
\node at (12.5,-1) {$a^{121}$};

\node at (-1,0.5) {$b^1$};
\node at (-1,1.5) {$b^2$};
\node at (-1,2.5) {$b^3$};
\node at (-1,3.5) {$b^4$};
\node at (-1,4.5) {$b^5$};
\node at (-1,5.5) {$\vdots$};
\node at (-1,6.5) {$b^{13}$};
\node at (-1,7.5) {$b^{14}$};
\node at (-1,8.5) {$\vdots$};
\node at (-1,9.5) {$b^{40}$};
\node at (-1,10.5) {$b^{41}$};
\node at (-1,11.5) {$\vdots$};
\node at (-1,12.5) {$b^{121}$};

\draw [Middle] (0+\x,0+\x) rectangle (1-\x,1-\x);
\node at (0.5,0.5) {$U_{ab}$};
\draw [BottomRight] (1+\x,0+\x) rectangle (4-\x,1-\x);
\node at (2.5,0.5) {$U_b V_a$};
\draw [TopLeft] (0+\x,1+\x) rectangle (1-\x,4-\x);
\node at (0.5,2.5) {$U_{a} V_b$};

\draw [Middle] (1+\x,1+\x) rectangle (4-\x,4-\x);
\node at (2.5,2.5) {$V_{ab}$};
\draw [Middle] (4+\x,4+\x) rectangle (7-\x,7-\x);
\node at (5.5,5.5) {$t V_{ab} T W_{ab}$};
\draw [Middle] (7+\x,7+\x) rectangle (10-\x,10-\x);
\node at (8.5,8.5) {$t^2 V_{ab} (T W_{ab})^2$};
\draw [Middle] (10+\x,10+\x) rectangle (13-\x,13-\x);
\node at (11.5,11.5) {$t^3 V_{ab} (T W_{ab})^3$};

\draw [TopLeft] (1+\x,4+\x) rectangle (4-\x,7-\x);
\node at (2.5,5.5) { $(t V_{b}) U_a (T W_{ab}) $};
\draw [TopLeft] (4+\x,7+\x) rectangle (7-\x,10-\x);
\node at (5.5,8.5) {$(t^2 V_{b}) U_a (TW_{ab})^2 $};
\draw [TopLeft] (7+\x,10+\x) rectangle (10-\x,13-\x);
\node at (8.5,11.5) {$(t^3 V_{b}) U_a(TW_{ab})^3 $};

\draw [BottomRight] (4+\x,1+\x) rectangle (7-\x,4-\x);
\node at (5.5,2.5) {$(t V_{a}) U_b (T W_{ab})$};
\draw [BottomRight] (7+\x,4+\x) rectangle (10-\x,7-\x);
\node at (8.5,5.5) {$(t^2 V_{a}) U_b (TW_{ab})^2$};
\draw [BottomRight] (10+\x,7+\x) rectangle (13-\x,10-\x);
\node at (11.5,8.5) {$(t^3 V_{a}) U_b (TW_{ab})^3$};

\draw [BottomRightFar] (7+\x,1+\x) rectangle (10-\x,4-\x);
\node at (8.5,2.5) {$(t^2 V_{a}) (T W_{a}) U_b  (TW_{ab})$};
\draw [BottomRightFar] (10+\x,1+\x) rectangle (13-\x,4-\x);
\node at (11.5,2.5) {$(t^3 V_a) (T W_a)^2 U_b (TW_{ab})$};
\draw  [BottomRightFar] (10+\x,4+\x) rectangle (13-\x,7-\x);
\node at (11.5,5.5) {$(t^3 V_a) (T W_a) U_b (TW_{ab})^2$};

\draw [BottomRightFar] (4+\x,0+\x) rectangle (7-\x,1-\x);
\node at (5.5,0.5) {$(t V_a) (T W_a) U_b$};
\draw [BottomRightFar] (7+\x,0+\x) rectangle (10-\x,1-\x);
\node at (8.5,0.5) {$(t^2 V_a) (T W_a)^2 U_b$};
\draw  [BottomRightFar] (10+\x,0+\x) rectangle (13-\x,1-\x);
\node at (11.5,0.5) {$(t^3 V_a) (T W_a)^3 U_b$};

\draw [TopLeftFar] (1+\x,7+\x) rectangle (4-\x,10-\x);
\node at (2.5,8.5) {$(t^2 V_b) (T W_b) U_a (T W_{ab})$};
\draw [TopLeftFar] (1+\x,10+\x) rectangle (4-\x,13-\x);
\node at (2.5,11.5) {$(t^3 V_b)  (T W_b)^2  U_a (T W_{ab})$};
\draw [TopLeftFar] (4+\x,10+\x) rectangle (7-\x,13-\x);
\node at (5.5,11.5) {$(t^3 V_b) (T W_b) U_a  (T W_{ab})^2$};

\draw [TopLeftFar] (0+\x,4+\x) rectangle (1-\x,7-\x);
\node [rotate=90] at (0.5,5.5) {$(tV_b)(T W_b) U_a$};
\draw [TopLeftFar] (0+\x,7+\x) rectangle (1-\x,10-\x);
\node [rotate=90] at (0.5,8.5) {$(t^2 V_b)(T W_b)^2 U_a$};
\draw [TopLeftFar] (0+\x,10+\x) rectangle (1-\x,13-\x);
\node [rotate=90] at (0.5,11.5) {$(t^3 V_b)(T W_b)^3 U_a$};



\end{tikzpicture}
\caption{Where the sets of $L_2$ evaluate to in $\Z_+^2 \subset \Z^2 *_{g \mapsto g^3}$. } \label{fig:Z2Words}
\end{figure}

\restoregeometry

\newpage

\section{ The Positive Series $P_m(x)$ and Rationality} \label{sec-Z3+} 
We have seen that the language $L_{m}$ is a language of geodesics for the positive orthant of $\Z^m$. 
Recall that its associated series is called the {\em positive series}, denoted $P_m(x)$.  (For brevity, we here abbreviate $U_m(x)$ etc.\ to $U_m$, etc.)

\begin{thm}
The growth series of the positive orthant, $P_m(x)$ is given by 

$$ P_m(x) = U_{m}  + \sum_{ i_1 + i_2 + \cdots + i_q = m  } \left(  {m \choose {i_1, \ldots, i_q }} ( V_{i_1}) (R_{i_1}) \left(  \prod_{1 < k < q}  (U_{i_k}) (R_{\sum_{ j \leq k} i_j} - 1) \right)  (U_{i_q})  (R_m) \right) $$

\end{thm}
\begin{proof}
From Theorem \ref{LanguageTheorem}, we have that the language $L_m$ given in Definition \ref{def:language} 
yields unique geodesic representatives of elements in the positive orthant.
For each fixed $k$, all permutations of $U_k$, $V_k$, and $R_k$ have the growth series $U_k(x)$, $V_k(x)$ and $R_k(x)$ respectively,
and the multinomial coefficient counts how many copies of each possible expression is achieved 
by counting the number of ways to partition the set of generators into subgroups of sizes given by the partition $(i_1, \ldots, i_q)$.

Notice that the formula works in the $m=1, 2$ and $3$ cases quite well, 
that the number of terms is equal  to one plus the number of ordered partitions of $m$,
which comes out to $1 + 2^{m-1}$.
For $m \geq 3$, we must remove the possibility of double counting the case where $U$ sets appear next to each-other,
which forces the middle terms of the product to have $(R(x) - 1)$. 
This insures we do not double count things of the form $U_{a} U_{b} = U_{ab}$, which can appear when $m > 2$ .
\end{proof}



\begin{cor}
For all dimensions $m$, the Positive Series $P_m(x)$ is rational.
\end{cor}

We can now put this together to prove our first theorem.

\begin{proof}[Proof of Theorem \ref{allrational}]
$\Z^m$ has $2m$ positive axis, $2(m)(m-1)$ positive coordinate planes, and in general ${m \choose i} 2^i$ positive $i$-dimensional orthants,
which have growth series series $P_i(x)$.
Combining those, we have
$$\Sm{m}(x) = 1 + 2m P_1(x) + 2(m)(m-1)P_2(x) + \ldots + 2^m P_m(x).$$
Since all $P_i(x)$ are rational, the subgroup growth series of $\Z^m$ is a finite sum of constant multiples of rational functions, 
and is thus rational.
\end{proof}

\begin{remark}
The prefix/suffix series and cap polynomials we computed in sections
\S \ref{sec-conj} and \S \ref{sec-cap} depend on our language
of geodesics for $BS(1,3)$.  $\Sm{m}(x)$ and $P_m(x)$ only depend on
the generating set.
\end{remark}
As $BS(1,3)$ is not uniquely geodesic, we could use a different
language of geodesics, and would then attain different $V_m$s, $R_m$s,
and $U_m$s.  However, when we put them together to get the $P_m(x)$
and $\Sm{m} (x)$, you would attain the same things.

\section {The Growth Series of the Full Group} \label{sec-fullgroup} 

Now that we have calculated the subgroup growth series, we use
recursive arguments to piece together copies of horocycles to complete
the proof of Theorem~\ref{grouprational}.  We follow the method of
\cite{FKS-GBSG1} (which covers the $m=1$ case) and \cite{FredenOHGGT}
(which leaves the $m=2$ case as an exercise for the reader).

\subsection{Relative growth} 
We start by defining `coset level' for our groups in the same manner:
Cosets with $\tau \geq 0$ are said to be at level $0$, and cosets with
$\tau < 0$ are said to be at level $\tau$.  This is equivalent to the
recursive definition that
\begin{itemize}
\item The horocyclic subgroup has level 0.
\item A coset directly above a coset of level 0 has level 0.
\item The coset $T\Z^m$ has level -1.
\item A coset directly below a coset of level $-n<0$ has level $-n-1$.
\item A coset directly above a coset of level $-n<0$ has level $-n+1$.
\end{itemize}

\begin{lem} ~

\begin{itemize}
\item For each coset $\gamma\Z^m$ there is a unique element $w$ which
  realizes the distance $d(1,\gamma\Z^m)$.
\item The geodesics for $w$ are unique up to commuting $\Z^m$
  generators.
\item The geodesic $w$ is empty or ends in $t$ or $T$.
\item The letter $T$ can only occur as a prefix $T^n$ of $w$.
\end{itemize}
\end{lem}

\begin{proof}
The third statement is immediate.  If $\gamma = \gamma' g$ with $g$ a
generator of $\Z^m$, then $\gamma'$ is a shorter and $\gamma'\Z^m =
\gamma\Z^m$.

To prove the first and second statements, recall that the cosets
$\gamma \Z^m$ form a tree.  For each coset, we define its {\em tree
  distance}, $d_\Tree(1,\gamma\Z^m)$ from the identity to be its
distance in the coset tree from the identity coset.  Note that
$d_\Tree(1,\gamma\Z^m)$ is rarely equal to $d(1,\gamma\Z^m)$.  For
example in the Baumslag-Solitar group, $d(1,at\Z) = 2$ while
$d_\Tree(1,at\Z)=1$.

We now induct on tree distance.  The statement is clearly true for the
horocyclic subgroup which is at tree distance 0.  We now consider a
coset $\gamma\Z^m$ at tree distance $d+1$ from the identity.  There is
a unique ray in the coset tree from the identity to $\gamma\Z^m$.
Suppose $\beta\Z^m$ is the unique coset on this ray at tree distance
$d$ from the identity.  By induction, there is a unique minimal length
coset representative which we may take to be $\beta$ and geodesics for
this element are unique up to commutation of letters in $\Z^m$. If
$\beta$ is the identity, we know that $\gamma\Z^m$ must be either $T\Z^m$
or $ut\Z^m$ where $u\in W_m$.  In these cases $T$ or $ut$ realizes the
unique shortest representative, and the geodesics for this element are
now unique up to commutation of the letters of $u$.

In the case where $\beta$ is not the identity, $\beta$ ends in either
$T$ or $t$.  In the first case $\gamma\Z^m = \beta T\Z^m$ and $\beta
T$ is the unique shortest coset element.  In the second case
$\gamma\Z^m = \beta u t \Z^m$ with $u\in W_m$.  Further, $\beta u t$ is
a geodesic spelling of the coset element closest to the identity and
once again unique up to commutation in $\Z^m$.

To prove the final statement notice that $T$ can not be preceded by
$t$ since this would not be reduced.  Neither can it be preceeded by a
generator $g \in \Z^d$ since $\gamma g T \Z^m = \gamma T \Z^m$, thus
producing a shorter coset representative.
\end{proof}

\begin{defn}
If $w$ is the shortest element of $w\Z^m$, we will say that $w$ is a
{\em stem}.  If $w$ is a stem, we will take $b(w,r)$ to be the number
of things in $w\Z^m$ whose length is $r + \ell(w)$.  We will call this
the {\em relative growth} of $w\Z^m$.
\end{defn}

\def\id{e}

It will turn out that this depends only on the level of $w$, and we
will supercede this notation by replacing $w$ with its level.
Previously, we have denoted the identity by 1.  However, in the
  upcoming change of notation, it will be replaced by its level, 0.
  Consequently, we now denote the identity by $\id$ to
  avoid confusion between numbers and group elements.

\begin{lem}\label{lem:above}
Suppose that $b(w,r) = b(\id ,r)$ and that $\gamma\in W_mt$, so that
$w\gamma\Z^m$ is immediately above $w\Z^m$.  Then $b(w\gamma,r) =
b(\id ,r)$.
\end{lem}

\begin{proof}
First, consider the case $\gamma = t$.  Every element element of
$w\Z^m$ has a geodesic $wx$ where $x \in L_m$.  This follows from the
fact that relative growth in this coset is the same as growth in the
subgroup.  Now each element of $L_m$ that has the form $tuTv$ with
$v\in W_m$ gives a geodesic $wtu$ for an element of $wt\Z^m$.  With the
exception of finitely many cases, this gives the elements of $L_m$,
showing that in these cases $wz$ has the same relative length as
$wtz$.  (Here $z\in \Z^m$.)  There remain finitely many cases where
the geodesic of an element does not begin with $t$.  These are the
elements of $\cup_m U_m^2$.  We claim that for $u\in U_m^2$, $w \gamma
t u$ is geodesic.  This is because in this case $tu$ cannot be
shortened either by writing $tu = t^2u'Tu''$ or by writing $tu =
u'tu''$ with $u' \ne \epsilon$.

We now consider the case where $\gamma = vt$ with $v\in W_m$.  We
consider those words of $L_m$ begining in $t$ which end in $v$ and
commute the $v$ to the front of each word.  The resulting word is
still geodesic and begins $vt$.  It therefore passes through the stem
of $w\gamma$ and we proceed as before.
\end{proof}

\begin{defn}
For $w$ a stem, we take $B_w(x) = \sum_{r=0}^\infty b(w,r) x^r$.
\end{defn}

It is the fate of $w$ to be replaced by its level in this notation as
well.

\begin{lem}\label{lem:levelDown}
$B_{T^n}(x) = (W_m(x))^n B_{\id}(x)$.
\end{lem}

\begin{proof}
The proof is an induction on $n$ and clearly holds when $n=0$.

Now suppose inductively that the formula holds for $n$.  Consider an
element $T^n g$.  Directly below it lies the element $T^ngT \in
T^{n+1}\Z^m$.  We then have $\ell(T^ngT) = \ell(T^n g)+1$, so these
two have the same relative length.  Further, for each $u \in W_m$ with
$\ell(u) = k$, we have $\ell(T^n g T u) = \ell(T^ngT)+k$.  It follows
that each element of $T^n \Z^m$ produces a $W_m$'s worth of elements in
$T^{n+1}\Z^m$, with relative lengths increased by the length of the
corresponding element of $W_m$.  This give $B_{T^{n+1}}(x) = W_m(x)
B_{T^n}(x)$ implying $B_{T^{n+1}}(x) = (W_m(x))^{n+1} B_{\id}(x)$ as required.
\end{proof}

\begin{lem}\label{lem:levelDownAndUp}
Suppose $j \le n$ and consider a stem $\gamma \in T^n (W_mt)^j$.  Then
$B_\gamma(x) = (W_m(x))^{n-j} B_{\id}(x)$.
\end{lem}

\begin{proof}
The proof is an induction on $j$ and is clearly true for $j=0$.

Now we suppose that the formula holds for $j-1$.  We can then write
$\gamma = \gamma'\delta$ where $\delta \in W_mt$.  As before, directly
under each element of $\gamma \Z^m$, there is one element of
$\gamma'\Z^m$ with the same relative length.  Likewise, for each
element of $\gamma \Z^m$, there is a $W_m$'s worth of elements in
$\gamma'\Z^m$, each with relative length increased by the length of
the corresponding element of $W_m$.  It follows that $B_{\gamma'}(x) =
W_m(x) B_\gamma(x)$.  By induction we have $B_{\gamma'}(x) =
(W_m(x))^{n-j+1}B_{\id}(x)$ and the result now follows. 
\end{proof}

\begin{cor}
If $\gamma \Z^m$ has level $-n$, then $B_\gamma(x) = (W_m(x))^n B_{\id}(x)$.
In particular we may drop the use of stems in this notation in favor
of their levels, restating this as $B_{-n}(x) = (W_m(x))^n B_0(x)$.
\end{cor}

\begin{proof}
For cosets at negative levels, this follows from Lemma~\ref{lem:levelDown}~and~\ref{lem:levelDownAndUp}, since every such coset is reached via a stem like those of Lemma~\ref{lem:levelDownAndUp} with $j<n$.  For cosets at level 0, it follows from Lemma~\ref{lem:levelDownAndUp} with $j=n$ together with Lemma~\ref{lem:above}.
\end{proof}

\subsection{Counting the cosets}

We now count the growth of the cosets.

\begin{defn}  
We take $\chi(-n,r)$ to count the number of cosets of level $-n\le 0$
at distance $r$ from the identity.  We take
$$ X_{-n}(x) = \sum_{r=0}^\infty \chi(-n,r) x^r.$$
\end{defn}

\begin{lem}
For $-n \le -1$, we have $\chi(-n-1,r+1) = \chi(-n,-r)$.  Hence, for $n\le -1$,  $X_{-n}(x) = x^{n-1} X_{-1}(x)$.
\end{lem}

\begin{proof}
To see this, decorate the tree of cosets with their levels and
distance from the identity.  Then notice that prepending $T$ carries
the cosets at level $-n$ to the cosets of level $-n-1$ while
increasing their length by 1.  It follows that for $-n \le -1$,
$$ X_{-n-1}(x) = x X_{-n}(x) $$ and hence for $-n \le -1$,
$$X_{-n}(x) = x^{n-1}X_{-1}(x).$$
\end{proof}

In what follows we will be using the observation that $\max(\{\ell(w)
\mid w \in W_m\})=m$.  

\begin{lem}
For $-n \le -1$,
\begin{align*}
X_{-1} &= \frac{p(x)}{1-x^2 W_m(x) } \\ X_{-n} &= \frac{p(x)
  x^{n-1}}{1-x^2 W_m(x) },
 \end{align*}
 where $p(x) $ is a polynomial defined below.
\end{lem}

\begin{proof}
Suppose that $\gamma\Z^m$ is a coset of level $-n-1 < -1$ at distance
$r$ from the identity.  Further, suppose $\gamma \ne T^{n+1}$, so that
$r>n+1$.  
Then for each element $w$ of $W_m$, there is a coset $\gamma w t$ of
level $-n$ at distance $r + \ell(w) +1$ from the identity.  It follows
that for $-n < -1$ and $r > n + m +2$, setting $n=1$ and
$p(x)$ to be the polynomial $p(x) = \sum_{r=1}^{m+3} \chi(-1,r)x^r$,
we have  
$$\chi(-n,r) = \sum_{w\in W_m} b(-n-1,r-\ell(w)-1) .$$ 
\begin{align*}
 \sum_{r=m+4}^\infty \chi(-n,r) x^r &= \sum_{r=m+4}^\infty\sum_{w\in
   W_m} b(-n-1,r-\ell(w)-1) x^r \\ X_{-1}(x) - p(x) &= x W_m(x) X_{-2}(x)
 \\ &= x^2 W_m(x) X_{-1}(x) \\ p(x) &= (1 - x^2 W_m(x) ) X_{-1}(x)
 \\ X_{-1}(x) &= \frac{p(x)}{1-x^2 W_m(x) } \\ X_{-n}(x) &=
 \frac{x^{n-1}p(x)}{1 - x^2 W_m(x) } \\
\end{align*}
The last holds for $n\ge 1$.
\end{proof}

We now turn to computing $\chi(0,r)$ and thus $X_0(x)$.  Observe that
for each coset of level 0 and distance $r$ from the identity and each
$w \in W_m$ there is a coset of level 0 at distance $r+\ell(w)+1$ from
the identity.  Likewise, for each coset of level $-1$ at distance $r
>1$ from the identity and each $w\in W_m$, there is a coset of level 0
at distance $r+\ell(w)+1$ from the identity. Take $q(x)$ to be the
polynomial $q(x) = \sum_{r=0}^{m+1} \chi(0,r)x^r$.  Then for $r \ge
m+2$, we have
\begin{align*}
\chi(0,r) &= \sum_{w\in W_m} \chi(0,r-\ell(w) -1) + \sum_{w\in W_m}
\chi(-1,r-\ell(w) -1) \\ 
\sum_{r=m+2}\chi(0,r) x^r&= \sum_{r=m+2}\sum_{w\in W_m} \chi(0,r-\ell(w) -1) x^r +
\sum_{r=m+2}\sum_{w\in W_m} \chi(-1,r-\ell(w) -1)x^r \\ 
X_0(x) - q(x) &= xW_m(x) X_0(x) + xW_m(x) X_{-1}(x) \\ 
X_0(x) &= \frac{x W_m(x) X_{-1}(x) +  q(x)}{1 - x W_m(x)}
\end{align*}

We are now able to give the growth of the group.

\begin{proof}[Proof of Theorem \ref{grouprational}]
 Take $\sigma(-n,r)$ to be the number of elements of level $-n$ at
 distance $r$ from the identity, and $S_{-n}(x) = \sum_{r=0}^\infty
 \sigma(-n,r) x^r$.  Then the growth of the group is given by
$$ \mathbb{S}(x) = \sum_{n=0}^\infty S_{-n}(x) = \sum_{n=0}^\infty X_{-n}(x) B_{-n}(x)$$ 
 Now $B_e(x) = B_0(x)$ is none other than the growth of the horocyclic subgroup $\Sm{m}(x)$,
 which we computed above and which by Theorem~\ref{allrational} is a
 rational function.  We have also computed the rational function
 $X_0(x)$.  In particular, the summand $X_0(x)B_0(x)$ is rational.
 
 On the other hand, we have expressions for $X_{-n}(x)=
 x^{n-1}X_{-1}(x)$ and $B_{-n}(x)=(W_m(x))^n B_0(x)$, for $n\ge 1$.  Thus

\begin{align*}
\mathbb{S}(x) &= S_0(x) +\sum_{n=1}^\infty S_{-n}(x) = \Sm{m}(x) X_0(x) +  \sum_{n=1}^\infty \Sm{m}(x) X_{-1}(x) \frac{x^{n} (W_m(x))^n}{x} \\
&= \Sm{m}(x) X_0(x) +  \Sm{m}(x) X_{-1}(x) \frac{ W_m(x)}{ 1 - x W_m(x)}  \\
&= \Sm{m}(x)  \frac{x W_m(x) X_{-1}(x) +  q(x)}{1 - x W_m(x)} +  \Sm{m}(x) X_{-1}(x) \frac{ W_m(x)}{1 - x W_m(x)} \\
&= \Sm{m}(x)  \frac{x W_m(x) X_{-1}(x) +  q(x) +  x W_m(x) X_{-1}(x) }{1 -  W_m(x)} \\
&= \Sm{m}(x)  \frac{x W_m(x) p(x)  +  q(x) (1 - x^2 W_m(x)) +   W_m(x) p(x) }{(1 - x W_m(x))(1 - x^2 W_m(x))}
\end{align*}  
  The result now follows.
 \end{proof}

Noting that $p(x) = \sum_{r=1}^{m+3} \chi(-1,r)x^r$, and $q(x) =
\sum_{r=0}^{m+1} \chi(0,r)x^r$, we can compute the growth series of
all groups in this class from some small initial data.
 For the case of $m = 1$, we compute it to be $\mathbb{S}(x) = \frac{(1+x)(1-x)^2 (1 + x + 2x^2)}{(1 - 2x)(1 - x^2 - 2x^3)^2}$, which matches with
the known growth series of $BS(1,3)$.  Similarly, for $m = 2$, we have
 $\mathbb{S}(x) = \frac{(1 -x)^2 (1 + x)(1 + 2x + 2x^2)^2(1 + 4x^2)}{(1  - 2x)^2 (1 + x + 2x^2)(1 - x - 4x^2 - 4x^3)}$, which matches with
\cite{FredenOHGGT} in 12.4 (under projects).


\vfill

\pagebreak

\section{Appendix} \label{sec-app}
\subsection{Subgroup Growth General Formulaions}
Below are the general formulas for the series for $V$,$R$, $P$, and $\Sm{m}$, as well as examples for $m \leq 3$.
\[\renewcommand{\arraystretch}{1.3}
 \begin{array}{lc}
\text{Growth Series} & \text{General Formula} \\ \hline \hline
W_m(x), \text{ Suffix Polynomial} & (1 + 2x)^m \\ 
V_{m}(x), \text{ Cap Polynomial} & x^{2m} + x^{m+2}W_m(x) - x^{2m+2}  \\
R_m(x), \text{ Prefix/Suffix Series} &  \left( 1 - x^2 W_m(x)  \right)^{-1} = \left( 1 -  \sum_{i=0}^m {m \choose i} 2^i x^{i+2}  \right)^{-1}   \\ 
P_m(x), \text{ Positive Series} & x^m  + \displaystyle\sum _{ i_1 + i_2 + \cdots + i_q = m  } 
\left(  \textstyle{m \choose {i_1, \ldots, i_q }}  V_{i_1}(x) x^{m-i_1} R_{i_1}(x) R_m(x)
 \displaystyle\prod_{1 < k < q}  \left(R_{i_1+\dots+i_k}(x) - 1\right) \right)    \\
\Sm{m} (x), \text{ Subgroup Growth Series} & 1 + \sum_{i=1}^m {m \choose i}  2^i  P_i(x) \\
\mathbb{S}(x), \text{Group Growth Series}&\displaystyle \Sm{m}(x)  \frac{x W_m(x) p(x)  +  q(x) (1 - x^2 W_m(x)) +   W_m(x) p(x) }{(1 - x W_m(x))(1 - x^2 W_m(x))} 
\end{array} \]

\[ \renewcommand{\arraystretch}{1.5}
\begin{array}{lcl}
\text{Series } & \text{Rational form}  & \text{Series expansion} \\ \hline \hline
V_{1}(x) & x^2 + x^3 + x^4 & x^2 + x^3 + x^4 \\  
V_{2}(x) & 2x^5 + 4x^5 + 3x^6 & 2x^5 + 4x^5 + 3x^6 \\ 
V_{3}(x) & x^5 + 7x^6 + 12x^7 + 7x^8  &  x^5 + 7x^6 + 12x^7 + 7x^8 \\  \hline

R_1(x) & \frac{1}{1-x^2 - 2x^3}  & 1 + x^2 + 2x^3 + x^4 + 4x^5 + 5x^6 + 6x^7 + \ldots  \\
R_2(x) & \frac{1}{1-x^2 - 4x^3 - 4x^4}  & 1 + x^2 + 4x^3 + 5x^4 + 8x^5 + 25x^6 + 44x^7 + \ldots  \\
R_3(x) & \frac{1}{1-x^2 - 6x^3 - 12x^4 - 8x^5}  & 1 + x^2 + 6x^3 + 13x^4 + 20x^5 + 61x^6 + 178x^7 + \ldots  \\ \hline

P_1(x) & \frac{x+x^2 -x^4}{1 - x^2 - 2x^3}  &  x + x^2 + x^3 +2x^4 + 3x^5 + 4x^6 + 7x^7 + \ldots  \\
P_2(x) & \frac{2 x^8-x^7-x^6-x^5+x^4+x^3+x^2}{\left(2 x^3+x^2-1\right) \left(4 x^3+x-1\right)}   &  x^2 + 2x^3 + 4x^4 + 10x^5 + 21 x^6 + 42 x^7 + \ldots  \\ 
P_3(x) & \frac{8x^{14} + \ldots - 2x^4 -  x^3}{64 x^{11} + \ldots + x - 1}  &  x^3 + 3x^4 + 10x^5 + 34x^6 + 94x^7 + 251x^8 +  \ldots  \\ \hline

\Sm{1} (x) & \frac{1 + 2x + x^2 - 2x^3 - 2x^4}{1-x^2-2x^3} & 1 + 2x + 2x^2 + 2x^3 + 4x^4 + 6x^5 + 8x^6 + 14x^7 + \ldots  \\ 
\Sm{2} (x) &   \frac{ (1 - x) (1 + 2 x + 2 x^2)^2}{(1 - 2 x) (1 + x + 2 x^2)} & 1 + 4 x + 8 x^2 + 12 x^3 + 24 x^4 + 52 x^5 + 100 x^6 + \ldots \\
\Sm{3} (x) & \frac{\left(x^2-1\right) \left(2 x^2+2 x+1\right)^3}{8 x^5+12 x^4+6 x^3+x^2-1}  & 1 + 6x + 18x^2 + 38x^3 + 84 x^4 + 218x^5 + 548x^6 + \ldots \\ \hline
\end{array} \]

\pagebreak 

\subsection{Subgroup growh tables for $m$ up to 10}
Below are the growth series for $Z^m$ in the HNN extension given by cubing for a few cases.  
The top table displays the growth series as a rational function, whereas the second table displays the first few terms of the series.

\[\renewcommand{\arraystretch}{1.8}
\begin{array}{lc}
\text{Case} & \text{Subgroup Growth Series $\Sm{m}(x)$ as a rational function} \\ \hline \hline
m = 1 & \frac{(-1 + x^2)(1 +2x + 2x^2)}{-1+x^2+2 x^3} \\
m = 2 & \frac{(-1+x) (1+2 x+2 x^2)^2}{-1+x+4 x^3} \\
m = 3 & \frac{(-1+x^2) (1+2 x+2 x^2)^3}{-1+x^2+6 x^3+12 x^4+8 x^5} \\
m = 4 & \frac{(-1+x) (1+2 x+2 x^2)^4}{-1+x+8 x^3+16 x^4+16 x^5} \\
m = 5 & \frac{(-1+x^2) (1+2 x+2 x^2)^5}{-1+x^2+10 x^3+40 x^4+80 x^5+80 x^6+32 x^7} \\
m = 6 & \frac{(-1+x) (1+2 x+2 x^2)^6}{-1+x+12 x^3+48 x^4+112 x^5+128 x^6+64 x^7} \\
m = 7 & \frac{(-1+x^2) (1+2 x+2 x^2)^7}{-1+x^2+14 x^3+84 x^4+280 x^5+560 x^6+672 x^7+448 x^8+128 x^9} \\
m = 8 & \frac{(-1+x) (1+2 x+2 x^2)^8}{-1+x+16 x^3+96 x^4+352 x^5+768 x^6+1024 x^7+768 x^8+256 x^9} \\
m = 9 & \frac{(-1+x^2) (1+2 x+2 x^2)^9}{-1+x^2+18 x^3+144 x^4+672 x^5+2016 x^6+4032 x^7+5376 x^8+4608 x^9+2304 x^{10}+512 x^{11}} \\
m = 10 & \frac{(-1+x) (1+2 x+2 x^2)^{10}}{-1+x+20 x^3+160 x^4+800 x^5+2560 x^6+5504 x^7+7936 x^8+7424 x^9+4096 x^{10}+1024 x^{11}} 
\end{array} \]

\[\renewcommand{\arraystretch}{1.5}
\begin{array}{ll}
\text{Case} & \text{Subgroup Growth Series } \Sm{m}(x)  \\ \hline \hline
m = 1 & 1+2 x+2 x^2+2 x^3+4 x^4+6 x^5+8 x^6+14 x^7+20 x^8+30 x^9+48 x^{10} + \ldots  \\
m = 2 & 1+4 x+8 x^2+12 x^3+24 x^4+52 x^5+100 x^6+196 x^7+404 x^8+804 x^9+1588 x^{10} + \ldots \\
m = 3 & 1+6 x+18 x^2+38 x^3+84 x^4+218 x^5+548 x^6+1298 x^7+3160 x^8+7874 x^9+19268 x^{10} + \ldots \\
m = 4 & 1+8 x+32 x^2+88 x^3+224 x^4+648 x^5+1960 x^6+5608 x^7+15736 x^8+45352 x^9+ \ldots \\
m = 5 & 1+10 x+50 x^2+170 x^3+500 x^4+1582 x^5+5440 x^6+18262 x^7+58860 x^8+190662 x^9 + \ldots \\
m = 6 & 1+12 x+72 x^2+292 x^3+984 x^4+3388 x^5+12812 x^6+48940 x^7+179372 x^8+647052 x^9 + \ldots \\
m = 7 & 1+14 x+98 x^2+462 x^3+1764 x^4+6594 x^5+26908 x^6+114074 x^7+468944 x^8+ \ldots \\
m = 8 & 1+16 x+128 x^2+688 x^3+2944 x^4+11920 x^5+51920 x^6+239696 x^7+1089840 x^8+ \ldots \\
m = 9 & 1+18 x+162 x^2+978 x^3+4644 x^4+20310 x^5+93816 x^6+465054 x^7+2309124 x^8+ \ldots \\
m = 10 & 1+20 x+200 x^2+1340 x^3+7000 x^4+32964 x^5+160820 x^6+847124 x^7+4542980 x^8+ \ldots 
\end{array} \]

\subsection{Programs}
Computations for $U_m$, $R_m$, $V_m$, $P_m$, and $\Sm{m}$ were carried out via the open source Python-based language SageMath.
The source code is available on the authors' (AS) website.
AS has also has made a Sage web applet to simplify spellings of words in the horocyclic subgroup for any HNN extension of 
$\Z^3$ which makes it easier to check what particular spellings evaluate to.


\end{document}